\documentclass{amsart}
\usepackage{graphicx}
\usepackage{amsfonts}
\newtheorem{tm}{Theorem}

\newtheorem{rems}{Remarks}
\newtheorem{lm}{Lemma}
\newtheorem{ex}{Example}

\newtheorem{prop}{Proposition}
\newtheorem{nota}{Notation}

\begin{document}

\title{On Descartes' rule of signs}
\author{Hassen Cheriha, Yousra Gati and Vladimir Petrov Kostov}
\address{Universit\'e C\^ote d'Azur, LJAD, France and 
University of Carthage, EPT - LIM, Tunisia}
\email{hassen.cheriha@gmail.com, hassan.cheriha@unice.fr}
\address{University of Carthage, EPT - LIM, Tunisia}
\email{yousra.gati@gmail.com} 
\address{Universit\'e C\^ote d'Azur, LJAD, France} 
\email{vladimir.kostov@unice.fr}

\maketitle 

\begin{abstract}
A sequence of $d+1$ signs $+$ and $-$ beginning with a $+$ is called  
a {\em sign pattern (SP)}. We say that the real polynomial 
$P:=x^d+\sum _{j=0}^{d-1}a_jx^j$, $a_j\neq 0$, 
defines the SP $\sigma :=(+$,sgn$(a_{d-1})$, $\ldots$, sgn$(a_0))$. 
By Descartes' rule of signs, for 
the quantity $pos$ of positive (resp. $neg$ of negative) roots of $P$, one has 
$pos\leq c$ (resp. $neg\leq p=d-c$), where $c$ and $p$ are the numbers of 
sign changes and sign preservations in $\sigma$; the numbers $c-pos$ 
and $p-neg$ are even. We say that $P$ realizes the SP $\sigma$ with the 
pair $(pos, neg)$. For SPs with $c=2$, we give some sufficient conditions for 
the (non)realizability of pairs $(pos, neg)$ of the form 
$(0,d-2k)$, $k=1$, $\ldots$, $[(d-2)/2]$. \\ 


\end{abstract}

\section{Introduction}

In the present paper we consider a problem which is a natural
continuation of Descartes' rule of signs. The latter states that
the number of positive roots of a real univariate polynomial
(counted with multiplicity) is majorized by the number of sign
changes in the sequence of its coefficients. We focus on polynomials
without zero coefficients. Such a polynomial (say, of degree $d$) is
representable in the form $P:=x^d+a_{d-1}x^{d-1}+\cdots +a_1x+a_0$,
$a_j\in \mathbb{R}^*$. Denoting by $c$ and $p$
the numbers of sign changes and sign
preservations in the sequence $1$, $a_{d-1}$, $\ldots$, $a_1$, $a_0$
and by $pos$ and $neg$ the number of positive and negative roots of $P$
(hence $c+p=pos+neg=d$) one obtains the conditions

\begin{equation}\label{Descartes}
\begin{array}{llllll}
pos\leq c&,&neg\leq p&,&c+p=d&,\\ \\
c-pos\in 2\mathbb{N}\cup 0&,&p-neg\in 2\mathbb{N}\cup 0&,&
(-1)^{pos}={\rm sgn}~(a_0)&
\end{array}
\end{equation}
(the condition $neg \leq p$ results from Descartes' rule applied to the
polynomial $P(-x)$).

We call {\em sign pattern (SP)} a sequence of $+$ or $-$ signs
of length $d+1$ beginning with a $+$. We say that the polynomial
$P$ defines the SP $(+$, sgn$(a_{d-1})$, $\ldots$, sgn$(a_1)$, sgn$(a_0))$.
A pair $(pos, neg)$ satisfying conditions (\ref{Descartes}) is called
{\em admissible}. An admissible pair (AP) is called {\em realizable} if 
there
exists a polynomial $P$ with exactly $pos$ positive distinct and exactly
$neg$ negative distinct roots.

\begin{ex}
{\rm For $c=0$, the all-pluses SP is realizable with any AP (which is of the 
form
$(0,d-2k)$, $k=0$, $1$, $\ldots$, $[d/2]$, where $[\alpha ]$ denotes the
integer part of $\alpha \in \mathbb{R}$). Indeed, one can construct a 
polynomial
$P$ with $d$ distinct negative roots and $d-1$ distinct critical levels.
Then in the family of polynomials $P+t$, $t>0$, one encounters 
polynomials
with exactly $d-2$, $d-4$, $\ldots$, $d-2[d/2]$ negative distinct roots 
and
with no positive roots (as $t$ increases, the polynomial $P+t$ loses two-by-two its real roots; each time two coalescing real roots give birth to a complex conjugate pair).}
\end{ex}

\begin{prop}\label{propc1}
(1) Any SP with $c=1$ is realizable with any AP of the form $(1,d-1-2k)$,
$k\leq [(d-1)/2]$. 

(2) Any SP is realizable with the AP $(c,p)$.
\end{prop}

Proposition~\ref{propc1} is proved in Section~\ref{secprpropc1}. Part (1) of it 
shows that in terms of the value of $c$, the first truly nontrivial case is $c=2$. Its study is the object of the present paper. We should point out that due to the possibility to 
consider instead of the polynomial $P(x)$ the polynomial $P(-x)$ (this 
change exchanges the quantities $c$ and $p$ and the quantities $pos$ 
and $neg$), it suffices to consider (for a given degree $d$) the cases with 
$c\leq [d/2]$. 

\begin{nota}\label{notaPR}
{\rm 
We denote by $\Sigma _{m,n,q}$ 
the SP consisting of $m\geq 1$ pluses followed by $n\geq 1$ minuses followed 
by $q\geq 1$ pluses, where $m+n+q=d+1$. For a given polynomial $P$, we denote by $P^R$ the corresponding  {\em reverted} polynomial, i.e. $P^R:=x^dP(1/x)$. If the polynomial $P$ defines the SP 
$\Sigma _{m,n,q}$, then $P^R$ defines the SP  $\Sigma _{q,n,m}$. The roots of $P^R$ are the reciprocals of the roots of $P$.}
\end{nota}

For small values of $m$ or $n$, we have the following result:

\begin{tm}\label{tmn=1234}
(1) For $n=1$, $d\geq 2$, and for $n=2$, $d\geq 3$, any SP $\Sigma _{m,n,q}$ 
is realizable
with the AP $(0,d-2)$.

(2) For $n=3$ and $d\geq 5$, any SP $\Sigma _{m,3,q}$ is realizable with 
the AP $(0,d-2)$. For
$d=4$, the SP $\Sigma _{1,3,1}$ is not realizable with the AP $(0,2)$.

(3) For $n=4$, the SP $\Sigma _{m,4,q}$ is realizable with the AP 
$(0,d-2)$ if $q\geq 3$, $m\geq 3$
and $d\geq 10$, or if $m=2$ and $q\geq 6$ (hence $d\geq 11$).

(4) For $m=1$ and $n\geq 4$, the SP $\Sigma _{1,n,q}$ 
is not realizable with the AP 
$(0,d-2)$.
\end{tm}

Theorem~\ref{tmn=1234} is proved in Section~\ref{secprtmn=1234}. 

\begin{rems}\label{remskappa}
{\rm (1) If a SP $\Sigma _{m,n,q}$ 
is realizable with the AP $(0,d-2)$, then it is realizable with 
any AP of the form $(0,d-2k)$, $k=1$, $\ldots$, $[(d-2)/2]$. 
Indeed, if a polynomial $P$ with distinct nonzero roots realizes 
the SP $\Sigma _{m,n,q}$, then one can perturb $P$ to make 
all its critical levels distinct. In the family $P+t$ 
one encounters (for suitable positive values of $t$) 
polynomials with exactly $d-2k$ distinct negative 
roots and no positive ones, for $k=1$, $\ldots$, $[(d-2)/2]$. 
As $t\geq 0$, the constant term of the polynomial $P$ is positive 
hence $P+t$ defines the SP $\Sigma _{m,n,q}$. 

(2) The exhaustive answer to the question which couples (SP, AP) 
are realizable for $d\leq 8$ is given in \cite{Gr},  \cite{AlFu}, \cite{FoKoSh} 
and \cite{KoCzMJ}. From the results in these papers one deduces that for $5\leq d\leq 8$, the SP $\Sigma _{m,4,q}$ is not realizable with the AP $(0,d-2)$. For $d\geq 9$, $n\leq 4$ and $c=2$, 
the only cases when the 
AP is $(0,d-2)$ and 
which are not covered by Theorem~\ref{tmn=1234} are the ones of 
$\Sigma _{3,4,3}$ and $\Sigma _{2,4,4}$ for $d=9$ and 
of $\Sigma _{2,4,5}$ for $d=10$. These cases are settled by Proposition~\ref{propthreecases}.

(3) The following result is proved in \cite{FoKoSh} (see Proposition~6 therein): {\em If $\kappa :=((d-m-1)/m)((d-q-1)/q)\geq 4$, then the SP $\Sigma _{m,n,q}$ is not realizable with the AP $(0,d-2)$.} This seems to be the only result concerning nonrealizability of the couple ($\Sigma _{m,n,q}$, $(0,d-2)$) known up to now. Part (4) of Theorem~\ref{tmn=1234} implies nonrealizability of cases which are not covered by the cited result. These are $\Sigma _{1,4,d-4}$ for $d\geq 11$ (with $\kappa =3(d-2)/(d-4)$ which is $\leq 27/7<4$ for $d\geq 11$). }
\end{rems}

\begin{prop}\label{propthreecases}
(1) For $d=9$, the SPs $\Sigma _{3,4,3}$ and $\Sigma _{2,4,4}$ are not realizable with 
the AP $(0,7)$.

(2) For $d=10$, the SP $\Sigma _{2,4,5}$ is not realizable with the AP $(0,8)$.
\end{prop}

Proposition~\ref{propthreecases} is proved in Section~\ref{secprpropthreecases}. 
Our next result 
contains sufficient conditions for realizability of a SP 
$\Sigma _{m,n,q}$ with the AP $(0,d-2)$:

\begin{tm}\label{Real}
The SP $\Sigma _{m,n,q}$ is realizable with 
the AP $(0,d-2)$ if 

\begin{equation}\label{eqsuffcond}\mathcal{L}(d,m,n):= -dn^2+4dm+4dn-4m^2-4mn-4d+4m>0~.  \end{equation}
\end{tm} 

Theorem~\ref{Real} is proved in Section~\ref{secprReal}. 

\begin{rems}{\rm (1) Condition (\ref{eqsuffcond}) is sharp in the following sense: in the two nonrealizable cases ($\Sigma _{3,4,3}$, $(0,7)$) and ($\Sigma _{2,4,5}$, $(0,8)$) (see Proposition~\ref{propthreecases}) one has $\mathcal{L}(d,m,n)=0$. 

(2) The condition of realizability (\ref{eqsuffcond}) can be compared with the condition of nonrealizability $\kappa \geq 4$ (see part (3) of Remarks~\ref{remskappa}). To this end the latter can be given the following equivalent form:}

$$3dm-dn-3m^2-3mn+2d+3m+n-2\leq 0~.$$

{\rm (3) As the SPs $\Sigma _{m,n,q}$ and $\Sigma _{q,n,m}$ are simultaneously (non)realizable with the AP $(0,d-2)$ (see the definition of $P^R$ in Notation~\ref{notaPR}), one can assume that $m\leq q$ hence $m\leq [d/2]$. Condition (\ref{eqsuffcond}) can be presented in the form 

$$\mathcal{M}d-\mathcal{N}m>0~~~\, {\rm with}~~~\, \mathcal{M}:=4m+4n-n^2-4~~~\, {\rm and}~~~\, \mathcal{N}:=4m+4n-4$$
which allows, for given $n=n_0$, to find $m_0$ such that for $m\geq m_0$, one has $\mathcal{N}/\mathcal{M}<2$. Then for $m\geq m_0$ and $d\geq 2m+n_0$, condition (\ref{eqsuffcond}) is fulfilled and the corresponding SP is realizable with the AP $(0,d-2)$.}
\end{rems}

\section{Proof of Proposition~\protect\ref{propc1}\protect\label{secprpropc1}}

Part (1). We remind the formulation of 
a {\em concatenation lemma} (see~\cite{FoKoSh}):

\begin{lm}\label{lmconcat}
Suppose that the
monic polynomials $P_1$ and $P_2$ of degrees $d_1$ and $d_2$ with SPs
$(+,\sigma _1)$ and $(+,\sigma _2)$ respectively realize
the pairs $(pos_1, neg_1)$ and $(pos_2, neg_2)$. Here $\sigma _j$
denote what remains of the SPs when the initial sign $+$ is deleted.
Then

(1) if the last position of $\sigma _1$ is $+$, then for any $\varepsilon >0$
small enough, the polynomial $\varepsilon ^{d_2}P_1(x)P_2(x/\varepsilon )$
realizes the SP $(1,\sigma _1,\sigma _2)$ and the pair
$(pos_1+pos_2, neg_1+neg_2)$;

(2) if the last position of $\sigma _1$ is $-$, then for any $\varepsilon >0$
small enough, the polynomial $\varepsilon ^{d_2}P_1(x)P_2(x/\varepsilon )$
realizes the SP $(1,\sigma _1,-\sigma _2)$ and the pair
$(pos_1+pos_2, neg_1+neg_2)$. Here $-\sigma _2$ is obtained from $\sigma _2$
by changing each $+$ by $-$ and vice versa.
\end{lm}

For $d=1$, the SP $(+,+)$ (resp. $(+,-)$) is realizable with the AP $(0,1)$ 
(resp. $(1,0)$) by the polynomial $x+1$ (resp. $x-1$). Applying 
Lemma~\ref{lmconcat} with $P_1$ and $P_2$ of the form $x\pm 1$ one realizes 
for $d=2$ all the three SPs with $c=0$ or $c=1$ with the APs of the form 
$(0,2)$ or $(1,1)$. 

Suppose that for $d=d_0\geq 2$ all SPs with $c=0$ or 
$c=1$ are realizable by monic polynomials (denoted by $P$). Then to realize 
for $d=d_0+1$ a SP with $c=0$ or $c=1$ with the pair $(0,d_0+1)$ or $(1,d_0)$ 
it suffices to apply Lemma~\ref{lmconcat} 
with $P_1=P$ and with $P_2=x-1$ (resp. $P_2=x+1$) if $c=0$ 
and the last two signs of the SP defined by 
$P$ are $(+,-)$ (resp. if $c=1$ and these last two signs are $(-,-)$).

To realize for $c=1$ a SP with any AP 
$(1,d-1-2k)$,
$k\leq [(d-1)/2]$, it suffices to perturb a polynomial $P$ realizing this SP 
with the pair $(1,d-1)$ so that all critical levels become distinct and then 
choose suitable values of $t>0$ in the family of polynomials $P-t$. 

To prove part (2) one has to apply $d-1$ times Lemma~\ref{lmconcat}. When it is applied for the first time one sets $P_1:=x-1$ (resp. $P_1:=x+1$) if the second entry of the SP is $-$ (resp. $+$). Each time the polynomial $P_2$ equals $x-1$ or $x+1$.

\section{Proof of Theorem~\protect\ref{tmn=1234}\protect\label{secprtmn=1234}}

Part (1). For $d=2$ and $d=3$, the polynomials

$$(x-1)^2+1=x^2-2x+2~~~\, ,~~~\, (x+2)((x-2)^2+2)=x^3-2x^2-2x+12$$
realize the APs $(0,0)$ and $(0,1)$ with the SPs $\Sigma _{1,1,1}$ and $\Sigma _{1,2,1}$ respectively. If a polynomial $P$ realizes a SP $\Sigma 
_{m,n,q}$ (with
$n=1$ or $2$) with the AP $(0,d-2)$, then the concatenation $Q$ of $P$ 
with $x+1$
realizes the SP $\Sigma _{m,n,q+1}$ with the AP $(0,d-1)$, and the 
polynomial $Q^R:=x^dQ(1/x)$ (the  reverted of $Q$) 
realizes the SP $\Sigma _{q+1,n,m}$ with the AP $(0,d-1)$. Thus by means 
of concatenation and
reversion one can realize all SPs $\Sigma _{m,1,q}$ and $\Sigma 
_{m,2,q}$ with the AP
$(0,d-2)$.

Part (2). For $d=4$, the nonrealizability of the SP $\Sigma _{1,3,1}$ with 
the AP $(0,2)$ is
proved in \cite{Gr}. For $d=5$, the SP $\Sigma _{1,3,2}$ is 
realizable with the AP
$(0,3)$, see \cite{AlFu}. To prove the first claim of part (2) one 
has to combine
concatenation and reversion as in the proof of part~(1) (applied to $\Sigma _{1,3,2}$).

Part (3). For $d=10$, the polynomial

$$\begin{array}{l}
(x+1)^8(x^2-2.49x+1.56)~=\\ \\ 
x^{10}+5.51x^9+9.64x^8-1.24x^7-25.76x^6-30.94x^5\\ \\ 
-2.24x^4+25.64x^3+24.76x^2+9.99x+1.56
 \end{array}$$
defines the SP $\Sigma _{3,4,4}$. (The quadratic factor is without real 
roots.)
One can perturb its $8$-fold root at $-1$ so that
the latter splits into $8$ negative simple roots. Thus the perturbation 
realizes
this SP with the AP $(0,8)$. Similarly, for $d=11$, a suitable 
perturbation of the
polynomial

$$\begin{array}{l}
(x+1)^9(x^2-4.69x+5.5)~=\\ \\
x^{11}+4.31x^{10}-0.71x^9-35.34x^8-69.96x^7-2.94x^6 \\ \\
+186.06x^5+335.04x^4+302.16x^3+156.79x^2+44.81x+5.5
 \end{array}$$
realizes the SP $\Sigma _{2,4,6}$ with the AP $(0,9)$. As in the proof 
of parts
(1) and (2), one deduces the realizability of all SPs as claimed by part 
(3) by
applying concatenation and reversion.

Part (4). Suppose that the SP $\Sigma _{1,n,q}$ with $n \ge 4$ 
is realizable by the polynomial 

$$P(x)= (x^{d-2}+e_1 x^{d-3}+\cdots +e_{d-2}) (x^2-zx+y)   $$
where $d\geq 5$, $z^2<4y$ and $e_j>0$ is the $j$th elementary symmetric function 
of the moduli $a_j$ of the negative roots of $P$.
(As $y>0$, the coefficient of $x$ of the quadratic 
factor must be negative, 
otherwise all coefficients of $P$ will be positive, so $z>0$).
Thus one obtains the conditions

$$
\begin{array}{lccccrccccc}z^2&<&4y&&&&&&&&,\\e_1-z&<&0&,&~~~{\rm i.~e.}\hspace{5mm}&
z&>&e_1&>&0&, \\e_2-e_1z+y&<&0&,&~~~{\rm i.~e.}\hspace{5mm}&
e_1z&>&y+e_2&&&,\\e_3-e_2z+e_1y&<&0&,&~~~{\rm i.~e.}\hspace{5mm}&
e_2z&>&e_3+e_1y&&&{\rm and}\\e_4-e_3z+e_2y&<&0&,&~~~{\rm i.~e.}\hspace{5mm}&
e_3z&>&e_4+e_2y&>&0&.\end{array}
$$
Keeping in mind that $e_j>0$, $y>0$ and $z>0$, one gets 

$$\begin{array}{ccccccccccc}z&>&(e_4+e_2y)/e_3&~~~\, {\rm and}~~~\,& z&<&2\sqrt{y}&&,&{\rm i.~e.}\\ \\ 
2e_3\sqrt{y}&>&e_4+e_2y &~~~\, {\rm and}~~~\,&
T(\sqrt{y})&:=&e_2y-2e_3\sqrt{y}+e_4&<&0&.&\end{array}$$ 
The quadratic polynomial $T$ has a positive discriminant $e_3^2-e_2e_4$ (this follows from Newton's inequalities). Hence it has two real roots, so the last 
inequality implies 
$$\sqrt{y}<\frac{e_3+\sqrt{e_3^2-e_2e_4}}{e_2}$$ 
and as $\displaystyle{\sqrt{y}>\frac{z}{2}>\frac{e_1}{2}}$, 
one deduces the condition 

\begin{equation}\label{ineq*} e_1e_2-2e_3<\sqrt{4e_3^2-4e_2e_4}~.  \end{equation}
One has $e_1e_2-2e_3>0$. Indeed, every product $a_ia_ja_k$ is encountered exactly two times in $2e_3$ and three times in $e_1e_2$ (and there are also the products $a_i^2a_j$ in $e_1e_2$). Hence one can take squares of both hand-sides of inequality $(\ref{ineq*})$ and then divide by $e_2$ to obtain the condition 
\begin{equation}\label{ineq**} e^2_1e_2+4e_4<4e_1e_3~. \end{equation}
We are going to show that for $d \ge 3$, 
\begin{equation}\label{ineq***} e^2_1e_2+4e_4>4e_1e_3 \end{equation}
which contradiction proves part (4). 
For $d=3$, one has $ e_1e_2 \ge 9e_3 $ (see Proposition 2 
on page 2 of \cite{Mitev}). Suppose that $(\ref{ineq***})$ holds true 
up to degree $d \ge3$. We proceed by induction on $d$. 
Recall that we denote by $(-a_j)$ the negative roots of $P$. 
For degree $d+1$, we have to show that 
$$\begin{array}{lc}
(a_{d-1}+e_1)^2 (a_{d-1}e_1+e_2)+4(a_{d-1}e_3+e_4)&>\\ \\ 
4(a_{d-1}+e_1)(a_{d-1}e_2+e_3)&,\end{array}$$ 
where $e_j$ are the elementary symmetric functions of the 
quantities $a_1, \ldots, a_{d-2}$, which is simplified to 
\begin{equation}\label{ineqA1}a_{d-1}^3e_1 +2a_{d-1}^2e_1^2+a_{d-1}e_1^3>
3a_{d-1}^2e_2+2a_{d-1}e_1e_2~. \end{equation}
Newton's inequality $\displaystyle{e_1^2\ge \frac{2d}{d-1} e_2}$ 
implies the following ones: 
\begin{equation}\label{ineqB}
\begin{array}{cccc}2a_{d-1}^2e_1^2&\ge&4a_{d-1}^2\frac{d}{d-1}e_2&{\rm and}\\ \\   
 a_{d-1}e_1^3&\ge&\frac{2d}{d-1} a_{d-1}e_1 e_2&\end{array} \end{equation}
From inequalities (\ref{ineqB}) we conclude that 

$$\begin{array}{ccl}a_{d-1}^3e_1 +2a_{d-1}^2e_1^2+
a_{d-1}e_1^3&>&4a_{d-1}^2\frac{d}{d-1}e_2 + 
\frac{2d}{d-1} a_{d-1}e_1 e_2\\ \\&>&          
3a_{d-1}^2e_2+2a_{d-1}e_1 e_2\end{array}$$
which proves (\ref{ineqA1}) and hence $(\ref{ineq***})$ as well.

\section{Proof of Proposition~\protect\ref{propthreecases}\protect\label{secprpropthreecases}}

We give in detail the proof of part (1). For part (2), we point out 
only the differences w.r.t. the proof of part (1). These differences 
are only technical in character. In order to give easily references to 
the different parts of the proof, the latter are marked by $1^0$, $2^0$, 
$\ldots$,~$6^0$.

\begin{proof}[Proof of part (1) of Proposition~\ref{propthreecases}]
%
$1^0$. Suppose that there exists a polynomial $P:=RQ$, 
where 

$$R:=(x+u_1)\cdots (x+u_7)~~~,~~~~u_j>0~~~~,~~~\, 
{\rm and}~~~\, Q:=x^2+rx+s~,$$   
which realizes one of the two SPs $\Sigma _{3,4,3}$ or $\Sigma _{2,4,4}$ 
with the AP $(0,7)$. We set 
$P:=\sum _{j=0}^9p_jx^j$ and $Q:=(x-a)^2+b$, $a\in \mathbb{R}$, $b\geq 0$. 
We show that 
for $b=0$, there exists no polynomial satisfying the conditions 

\begin{equation}\label{twoineq}p_3<0~~~,~~~ 
p_6<0~~~{\rm ,~~~\, resp.}~~~\, p_4<0~~~,~~~p_7<0~.
\end{equation} 
Hence this holds true also for $b>0$ because 
$P=R\cdot Q|_{b=0}+bR$, and the polynomial $R$ has all coefficients positive. 
This in turn implies that for $b\geq 0$, there exists no polynomial  
$P$ realizing the SP $\Sigma _{3,4,3}$ or 
$\Sigma _{2,4,4}$. So from now on we concentrate on the case $b=0$.

$2^0$. Suppose that a polynomial $P$ with $b=0$ and 
$u_1\geq u_2\geq \cdots \geq u_7\geq 0$ satisfying the left or right 
couple of inequalities (\ref{twoineq}) exists. We make the change of variables 
$x\mapsto u_1x$ and after this we multiply $P$ by $(1/u_1)^9$ (these changes 
preserve the signs of the coefficients), 
so now we are in the case $u_1=1$. Denote by 
$\Delta\subset \mathbb{R}_+^7=\{ (u_2,u_3,\ldots ,u_7,a)\}$ the set  
on which one has conditions (\ref{twoineq}). The closure $\overline{\Delta}$ 
of this set is compact. Indeed, one has $p_1\geq 0$ hence 

$$1+u_2+\cdots +u_7-2a\geq 0~~~\, {\rm and}~~~\, u_j\leq 1~~~\, 
{\rm hence}~~~\, a\in [0,7/2]~.$$
The set $\overline{\Delta}$ can be stratified according to the 
multiplicity vector of the variables $(u_2,\ldots ,u_7)$ 
and the possible equalities $u_j=0$, $u_i=1$ and/or $a=0$. 
Suppose that the set $\overline{\Delta}$ 
contains a polynomial satisfying the inequalities (\ref{twoineq}). 

\begin{rems}\label{remsDelta} 
{\rm (1) For this polynomial one has $a>0$, otherwise all its coefficients 
are nonnegative. One has also $u_j>0$, $j=2$, $\ldots$, $7$. Indeed, in 
the case of $\Sigma _{3,4,3}$ (resp. $\Sigma _{2,4,4}$), 
if three or more (resp. if four or more) 
of the variables $u_j$ are $0$, then the polynomial $P$ has 
less than two sign changes in the sequence of its coefficients and by the 
Descartes rule of signs $P$ cannot have two positive roots counted with 
multiplicity. For $\Sigma _{3,4,3}$, if exactly one or two  
of the variables $u_j$ equal $0$, then 
the polynomial $P$ is the product of $x$ with a polynomial defining the 
SP $\Sigma _{3,4,2}$ or of $x^2$ with a polynomial defining the 
SP $\Sigma _{3,4,1}$. However these SPs are not realizable with the APs 
$(0,6)$ or $(0,5)$ respectively, see \cite{KoCzMJ} and \cite{FoKoSh}. For $\Sigma _{2,4,4}$, if exactly 
one, two or three of the variables $u_j$ equal $0$, then $P$ is the product of 
$x$, $x^2$ or $x^3$ with a polynomial defining respectively the SP 
 $\Sigma _{2,4,3}$, $\Sigma _{2,4,2}$ or $\Sigma _{2,4,1}$ which 
is not realizable with the AP $(0,6)$, $(0,5)$ or $(0,4)$, see \cite{KoCzMJ}, \cite{FoKoSh} and~\cite{AlFu}. 

(2) The set $\overline{\Delta}$ being compact the quantity $p_3+p_6$, 
resp. $p_4+p_7$, attains its minimum $-\delta$ on it ($\delta >0$). 
Consider the set 
$\Delta ^{\bullet}\subset \overline{\Delta}$ on which one has 
$p_3+p_6\leq -\delta /2$, resp. $p_4+p_7\leq -\delta /2$. On this set one has 
$a\geq 2^{-9}\delta$. Indeed, $P=x^2R-2axR+a^2R$, 
so any coefficient of $P$ is not less than 
$-2a\sigma$, where $\sigma$ is the sum of all coefficients of $R$ 
(they are all nonnegative); clearly 
$\sigma \leq 2^7$ (follows from $u_j\in [0,1]$). 

(3) There exists $\delta _*>0$ such that on the set $\Delta ^{\bullet}$,  
one has also $u_j\geq \delta _*$. This follows from part (1) 
of the present remarks.}
\end{rems}

$3^0$. We need some technical lemmas:

\begin{lm}\label{lmnotfour}
The minimum of the quantity $p_3+p_6$, resp. $p_4+p_7$, is not attained 
at a point of the set $\Delta ^{\bullet}$ with three or more distinct 
and distinct from $1$ among the 
quantities $u_j$, $2\leq j\leq 7$.
\end{lm}

The lemmas used in the proof of part (1) of Proposition~\ref{propthreecases} 
are proved after the proof of part~(1). 

\begin{lm}\label{lmnottwo}
Conditions (\ref{twoineq}) fail for $u_1=u_2=\cdots =u_7=1$ and any $a>0$.
\end{lm}

Thus to prove Proposition~\ref{propthreecases} we have to consider 
only the case when exactly one or two of the quantities 
$u_j$ are distinct from $1$. We use the following result:

%

\begin{lm}\label{lmdiscrim}
For $d\geq 4$, set $P:=RQ$, where $R:=\prod _{i=1}^{d-2}(x+u_i)$, $u_i>0$, 
$Q:=(x-a)^2$. Then the coefficients $p_j$ of $P$, $j=2$, $\ldots$, $d-2$,  
are quadratic polynomials in $a$ 
with positive leading coefficients and with two distinct positive roots.
\end{lm} 

$4^0$. Further we consider several different cases 
according to the multiplicity of $u_{j_0}$, 
the smallest of the variables $u_j$. 
In the proofs we use linear changes $x\mapsto \chi x$, 
$\chi >0$, followed by $P\mapsto \chi ^{-9}P$. These changes preserve 
the signs of the coefficients; the condition $u_1=1$ is lost and the 
condition $u_{j_0}=1$, $j_0\neq 1$, is obtained. The aim of this is to have 
more explicit computations. In all the cases the polynomial $R$ is of the form 
$R=(x+1)^{s_1}(x+v)^{s_2}(x+w)^{s_3}$, $s_1+s_2+s_3=7$, and one has 
$v>1$, $w>1$, but $v$ and $w$ are not necessarily distinct and we do not 
suppose that $v>w$ or $v<w$ (which permits us to assume that $s_2\geq s_3$). Allowing the 
equality $v=w$ means treating together cases of exactly two or exactly 
three distinct 
quantities $u_j$ (counting also $u_1=1$). 
We list the triples $(s_1,s_2,s_3)$ defining the cases:

$$\begin{array}{cccccccccc}
(5,1,1)&,&(4,2,1)&,&(3,3,1)&,&(3,2,2)&,&(2,4,1)&,\\ \\ 
(2,3,2)&,&(1,5,1)&,&(1,4,2)&&{\rm and}&&(1,3,3)&.\end{array}$$
The cases when there are exactly two different quantities $u_j$ one of which is 
$u_1=1$ can be coded in a similar way. E.g. $(5,2)$ means that 
$R=(x+1)^5(x+u)^2$, $u>1$. The nonrealizability of these cases 
follows automatically from the one of the above $9$ ones 
(when $v$ and $w$ coalesce), 
with the only exception of $R=(x+1)^6(x+w)$ (the case $(6,1)$). 

\begin{lm}\label{lm61}
Conditions (\ref{twoineq}) fail in case $(6,1)$.
\end{lm}

$5^0$. We consider the SP $\Sigma _{2,4,4}$ first. We compute using MAPLE the resultant 
Res\,$(p_4,p_7,a)$ as a function of $v$ and $w$. Then we set 
$v:=1+V$, $w:=1+W$, $V>0$, $W>0$. In all $9$ cases this resultant is a 
polynomial in $V$ and $W$ with all coefficients positive. Hence for no 
value of $V>0$ and $W>0$ do the coefficients $p_4$ and $p_7$ vanish together. 

In all $9$ cases, the leading coefficients of $p_4$ and $p_7$ considered as 
quadratic polynomials in $a$ are positive. In fact, they are polynomials in 
$v$ and $w$ with all coefficients positive. 
For $v=w=2$, we compute the two roots $y_1<y_2$ of $p_4$ and the two roots 
$y_3<y_4$ of $p_7$. In all $9$ cases, one has $y_1<y_2<y_3<y_4$. By continuity, 
these inequalities hold true for all values of $v>1$ and $w>1$. Hence 
the intervals $(y_1,y_2)$ and $(y_3,y_4)$ on which $p_4$ and $p_7$ are negative 
do not intersect for any $v>1$, $w>1$. This proves the proposition in the case 
of $\Sigma _{2,4,4}$. 

$6^0$. Consider now the SP $\Sigma _{3,4,3}$. Recall that the polynomials $P(x)$ and 
$x^9P(1/x)$ have one and the same numbers of positive and negative roots. 
Their roots are mutually reciprocal and they define the same SP. 
Hence the non-realizability of the case 
$(5,1,1)$ (resp. $(4,2,1)$, or $(3,3,1)$, or $(3,2,2)$) implies the one of 
$(1,5,1)$ (resp. $(2,4,1)$ and $(1,4,2)$, or $(1,3,3)$, or $(2,3,2)$). 

As in the case of $\Sigma _{2,4,4}$, we express Res\,$(p_3,p_6,a)$ 
as a polynomial of $v$ and $w$, and then of $V$ and $W$. In cases 
$(5,1,1)$, $(4,2,1)$ and $(3,2,2)$, this resultant has a single monomial with 
negative coefficient, this is $UV$. We give the monomials $VW$, $V^2$ and $W^2$ 
for these three cases:

$$\begin{array}{lll}
(5,1,1)&-9408VW+28224V^2+28224W^2&,\\ \\ 
(4,2,1)&-18816VW+47040V^2+28224W^2&,\\ \\ 
(3,2,2)&-37632VW+47040V^2+47040W^2&.\end{array}$$
The discriminants of these quadratic homogeneous 
polynomials are negative hence they are nonnegative (and positive for $V>0$, 
$W>0$). In the case of $(3,3,1)$, there are exactly two monomials 
with negative coefficients, namely $VW$ and $V^2W$. The resultant equals 

$$(-28224VW+56448V^2+28224W^2)+V(-42336VW+127008W^2+282240V^2)+\cdots$$
(we skip all other monomials; their coefficients are positive). The two 
quadratic homogeneous polynomials have negative discriminants, so they are 
positive for $V>0$, $W>0$.  

The rest of the reasoning goes by exact analogy with the case of 
$\Sigma _{2,4,4}$.

\end{proof}

\begin{proof}[Proof of Lemma~\ref{lmnotfour}]
Denote by $v_1$, $v_2$ and $v_3$ three distinct and distinct from $1$ 
of the variables $u_j$. We prove that one can choose 
$v_1^*$, $v_2^*$, $v_3^*$, $a^*\in \mathbb{R}$ such that 
the infinitesimal change $v_j\mapsto v_j+\varepsilon v_j^*$, $j=1$, $2$, $3$,  
$a\mapsto a+\varepsilon a^*$, $\varepsilon >0$, 
results in 
$p_{\mu}\mapsto p_{\mu}+\varepsilon p_{\mu}^*+o(\varepsilon )$, 
$p_{\nu}\mapsto p_{\nu}+\varepsilon p_{\nu}^*+o(\varepsilon )$, where 
$(\mu ,\nu )=(3,6)$ or $(4,7)$ and $p_{\mu}^*<0$, $p_{\nu}^*<0$. Hence locally 
the quantity $p_{\mu}+p_{\nu}$ is not minimal. 

Set $P:=(x+v_1)^{\alpha _1}(x+v_2)^{\alpha _2}(x+v_3)^{\alpha _3}(x-a)^2P^{\dagger}$, 
where $a$, $-v_1$, $-v_2$ and 
$-v_3$ are not roots of $P^{\dagger}$ and $\alpha _j$ are the multiplicities of 
throots $-v_j$ of $P$. Set $P_{v_j}:=P/(x+v_j)$, 
$P_a:=P/(x-a)$, $P_{v_i,v_j}:=P/((x+v_i)(x+v_j))$, $P_{a,v_j}:=P/((x+a)(x+v_j))$ 
etc. Then the above infinitesimal change transforms $P$ into 

$$P+\varepsilon \tilde{P}+o(\varepsilon )~,~~~\, {\rm where}~~~\,  
\tilde{P}:=\sum _{j=1}^3\alpha _jv_j^*P_{v_j}-2a^*P_a~.$$   
We show that one can choose $v_j^*$ and $a^*$ such that the 
coefficients of $x^{\mu}$ and $x^{\nu}$  
of the polynomial $\tilde{P}$ (where $(\mu ,\nu )=(3,6)$ or $(4,7)$) are both 
negative from which the lemma follows. 
To this end we observe that each of the polynomials $P_{v_j}$ 
and $P_a$ is a linear combination of 
$P^{\diamond}:=P_{v_1,v_2,v_3,a}:=x^5+Ax^4+Bx^3+Cx^2+Dx+E$, 
$xP^{\diamond}$, $x^2P^{\diamond}$ and $x^3P_{v,w,a}$.  

We consider first the case of $\Sigma _{3,4,3}$, i.e $(\mu ,\nu )=(3,6)$. 
The $2$-vectors of coefficients of $x^3$ and $x^6$ of the polynomials 
$P^{\diamond}$, $xP^{\diamond}$, $x^2P^{\diamond}$ and $x^3P^{\diamond}$ equal  
$(B,0)$, $(C,1)$, $(D,A)$ and $(E,B)$ respectively. For $B\neq 0$, 
the first two 
of them are not collinear. As $E\neq 0$ (see parts (2) and (3) 
of Remarks~\ref{remsDelta}), 
for $B=0$, the second and fourth of these vectors are not 
collinear and the choice of $v_j^*$ and $a^*$ is possible. 

If $(\mu ,\nu )=(4,7)$, then the $2$-vectors of coefficients of $x^4$ and $x^7$ 
equal $(A,0)$, $(B,0)$, $(C,1)$ and $(D,A)$. One has either 
$A\neq 0$ or $B\neq 0$. Indeed, the polynomial $P^{\diamond}$ has all roots real 
and by Rolle's theorem this is the case of $(P^{\diamond})'$ and 
$(P^{\diamond})''$ as well. If $A=B=C=D=0\neq E$ 
(resp. $A=B=C=0\neq D$ or $A=B=0\neq C$), then $P^{\diamond}$ (resp. 
$(P^{\diamond})'$ or $(P^{\diamond})''$) has not all roots real. Thus either 
$(A,0)$, $(C,1)$ or $(B,0)$, $(C,1)$ are not collinear and 
the choice of $v_j^*$ and $a^*$ is possible.

\end{proof}

\begin{proof}[Proof of Lemma~\ref{lmnottwo}]

For the polynomial $(x+1)^7(x-a)^2$, we list its coefficients 
$p_3$, $p_4$, $p_6$ and $p_7$ and their roots: 

$$\begin{array}{cclccclc}
p_3&=&7-42a+35a^2&,&p_4&=&21-70a+35a^2&,\\ \\ 
0.2&,&1&&0.36\ldots&,&1.63\ldots&\\ \\ 
p_6&=&35-42a+7a^2&,&p_7&=&21-14a+a^2&.\\ \\ 
1&,&5&&1.70\ldots&,&12.2\ldots&\end{array}$$
Hence for no value of $a\geq 0$ does one have the left or the right two of 
conditions (\ref{twoineq}) together.

\end{proof}

\begin{proof}[Proof of Lemma~\ref{lmdiscrim}]

Set $R:=r_{d-2}x^{d-2}+r_{d-3}x^{d-3}+\cdots +r_0$, $r_j>0$, $r_{d-2}=1$. 
The polynomial $R$ has $d-2$ negative 
roots. Hence Newton's inequalities hold true:

\begin{equation}\label{eqNewton}
\left( r_k/{d-2\choose k}\right) ^2\geq 
\left( r_{k-1}/{d-2\choose k-1}\right) \left( r_{k+1}/{d-2\choose k+1}\right) 
~~~,~~~k=1,\ldots ,d-3~.
\end{equation}
The coefficient $p_{k+1}$ equals $a^2r_{k+1}-2ar_k+r_{k-1}$, 
$k=1$, $\ldots$, $d-3$, $r_{k+1}>0$. 
This quadratic polynomial has two distinct positive roots if and only if 
$r_k^2>r_{k-1}r_{k+1}$. These inequalities result from (\ref{eqNewton}) because 
${d-2\choose k}^2>{d-2\choose k-1}{d-2\choose k+1}$ (the latter inequality is 
equivalent to $((k+1)/k)((d-1-k)/(d-2-k))>1$ which is true). 

\end{proof}

\begin{proof}[Proof of Lemma~\ref{lm61}]
In case (6,1), with $P=(x+1)^6(x+w)(x-a)^2$, one has 

$$\begin{array}{lllc}
p_3&=&1+6w-12a-30wa+15a^2+20wa^2&,\\  p_4&=&6+15w-30a-40wa+20a^2+15wa^2&,\\ 
p_6&=&20+15w-30a-12wa+6a^2+wa^2&{\rm and}\\ p_7&=&15+6w-12a-2wa+a^2&.\end{array}
$$
For $w=1$, the roots of $p_4$ (resp. of  $p_7$) equal $0.36\ldots$ and $1.63\ldots$ (resp. $1.70\ldots$ and $12.29\ldots$). As Res\,$(p_4,p_7,a)=7056+2520w+540w^2+3960w^3+1800w^4$ has no positive roots, for any $w>0$ fixed, the two intervals of values of $a$, for which $p_4<0$ or $p_7<0$, do not intersect. Hence the couple of conditions $p_4<0$, $p_7<0$ fails.

One has Res\,$(p_3,p_6,a)=7056(w-1)^2(w+1)^2$, so only for $w=1$ do the polynomials $p_3$ and $p_6$ have a root in common. For $w=1/2$, $w=1$ and $w=2$, the roots of $p_3$ and $p_6$ equal respectively

$$\begin{array}{llllclllc}
w=1/2&0.17\ldots&,&0.90\ldots&{\rm and}&0.91\ldots&,&4.62\ldots&;\\ 
w=1&0.2&,&1&{\rm and}&1&,&5&;\\ 
w=2&0.21\ldots&,&1.09\ldots&{\rm and}&1.10\ldots&,&5.64\ldots&.
\end{array}$$
Hence again the intervals of values of $a$ for which $p_3<0$ or $p_6<0$ do not intersect and the couple of conditions $p_3<0$, $p_6<0$ fails.

\end{proof}

\begin{proof}[Proof of part (2) of Proposition~\ref{propthreecases}]
$1^0$. In the analog of part $1^0$ of the proof of part (1), 
we set $R:=(x+u_1)\cdots (x+u_8)$, $u_j>0$, and the analog of 
inequalities (\ref{twoineq}) reads $p_5<0$, $p_8<0$.

$2^0$. In the analog of part $2^0$ we make the change of variables 
$x\mapsto u_1x$ and then we multiply $P$ by $(1/u_1)^{10}$. We denote by 
$\Delta\subset \mathbb{R}_+^8=\{ (u_2,u_3,\ldots ,u_8,a)\}$ the set  
on which one has the conditions $p_5<0$, $p_8<0$. 
On the closure $\overline{\Delta}$ 
of this set one has $p_1\geq 0$ hence 

$$1+u_2+\cdots +u_8-2a\geq 0~~~\, {\rm and}~~~\, u_j\leq 1~~~\, 
{\rm hence}~~~\, a\in [0,4]~.$$
The analog of Remarks~\ref{remsDelta} reads:

\begin{rems}\label{remsDeltabis} 
{\rm (1) One has $u_j>0$, $j=2$, $\ldots$, $8$. Indeed, if 
exactly one of the quantities $u_j$ is $0$, then $P=xY$, where 
the polynomial $Y$ defines the SP $\Sigma _{2,4,4}$ which by 
part (1) of Proposition~\ref{propthreecases} is impossible. If more than one of 
the quantities $u_j$ is $0$, then see part (1) of Remarks~\ref{remsDelta} 
about $\Sigma _{2,4,4}$.   

(2) In the proof of part (2) of Proposition~\ref{propthreecases} we 
define the set $\Delta ^{\bullet}\subset \overline{\Delta}$ 
as the one on which one has 
$p_5+p_8\leq -\delta /2$. On this set one has 
$a\geq 2^{-10}\delta$. Indeed, as $P=x^2R-2axR+a^2R$, 
any coefficient of $P$ is not less than 
$-2a\sigma$, where $\sigma$ is the sum of all coefficients of $R$ 
(they are all nonnegative); clearly 
$\sigma \leq 2^8$ (follows from $u_j\in [0,1]$). }
\end{rems}

$3^0$. The analog of Lemma~\ref{lmnotfour} reads: {\em 
The minimum of the quantity $p_5+p_8$ is not attained
at a point of the set $\Delta ^{\bullet}$ with three or more distinct 
and distinct from $1$ among the
quantities $u_j$, $2\leq j\leq 8$.}

The proof is much the same as the one of Lemma~\ref{lmnotfour}. One sets 
$(\mu ,\nu )=(5,8)$. Each of the polynomials 
$P_{v_j}$ and $P_a$ is a linear combination of 
$P^{\diamond}:=P_{v_1,v_2,v_3,a}:=x^6+Ax^5+Bx^4+Cx^3+Dx^2+Ex+F$, 
$xP^{\diamond}$, $x^2P^{\diamond}$ and $x^3P_{v,w,a}$. The 
$2$-vectors of coefficients of $x^5$ and $x^8$ of the polynomials 
$P^{\diamond}$, $xP^{\diamond}$, $x^2P^{\diamond}$ and $x^3P^{\diamond}$ equal
$(A,0)$, $(B,0)$, $(C,1)$ and $(D,A)$ respectively. If $A\neq 0$ or 
$B\neq 0$, there are two noncollinear among the first three of these vectors
and the choice of $v_j^*$ and $a^*$ is possible. If $A=B=0$, then, as 
$F\neq 0$, either the polynomial $P^{\diamond}$ or one of its derivatives 
is not with all roots real which is a contradiction.

The analog of Lemma~\ref{lmnottwo} reads: {\em Conditions 
$p_5<0$, $p_8<0$ fail for $u_1=\cdots =u_8=1$ and any $a>0$.} 

Here's the proof of this. For the polynomial $(x+1)^8(x-a)^2$, we 
list its coefficients $p_5$, $p_8$ and their roots: 

$$\begin{array}{cclccclc}
p_5&=&28(2-5a+2a^2)&,&p_8&=&28-16a+a^2&,\\ \\ 
0.5&,&2&&2&,&14& 
\end{array}$$
Hence for no value of $a\geq 0$ does one have $p_5<0$, $p_8<0$.

We remind that Lemma~\ref{lmdiscrim} is formulated for any $d\geq 4$.

$4^0$. In the analog of part $4^0$ of the proof, one has 
$R=(x+1)^{s_1}(x+v)^{s_2}(x+w)^{s_3}$, $s_1+s_2+s_3=8$, and one has to 
consider the following cases of exactly three different quantities~$u_j$:

$$\begin{array}{cccccccccccc}
(6,1,1)&,&(5,2,1)&,&(4,3,1)&,&(4,2,2)&,&(3,4,1)&,&(3,3,2)&,\\ \\ 
(2,5,1)&,&(2,4,2)&,&(2,3,3)&,&(1,6,1)&,&(1,5,2)&{\rm and}
&(1,4,3)&.\end{array}$$
The cases with exactly two different quantities $u_j$ are treated in the 
same way. The exceptional case is the one with $R=(x+1)^7(x+w)$ 
(the case $(7,1)$).

\begin{lm}\label{lm71}
The conditions $p_5<0$, $p_8<0$ fail in case $(7,1)$.
\end{lm}

\begin{proof}
Set $P:=(x+1)^7(x+w)(x-a)^2$. Then 

$$p_5=21+35w-70a-70wa+35a^2+21wa^2~~~\, {\rm and}~~~\, p_8=21+7w-14a-2wa+a^2~.$$
One has Res\,$(p_5,p_8,a)=3969(w+2)^2(w-1)^2$. We list the roots of $p_5$ and $p_8$ for $w=1/2$, $w=1$ and $w=2$:

$$\begin{array}{llllcllll}w=1/2&0.45\ldots&,&1.850\ldots&{\rm and}&1.865\ldots&,&13.13\ldots &;\\ w=1&0.5&,&2&{\rm and}&2&,&14&;\\ w=2&0.54\ldots&,&2.18\ldots&{\rm and}&2.21\ldots&,&15.78\ldots&. \end{array}$$
As in the proof of Lemma~\ref{lm61}, we conclude that the conditions $p_5<0$, $p_8<0$ fail for $w>0$.
\end{proof}

$5^0$. We compute Res\, $(p_5,p_8,a)$ as a function of $v$ and $w$ and then 
set $v:=1+V$, $w:=1+W$. Our aim is to show that in all $12$ cases, 
the leading coefficients of 
$p_5$ and $p_8$ considered as 
quadratic polynomials in $a$ are positive. The rest of the reasoning is 
done by analogy with part $5^0$ of the proof of part (1) of 
Proposition~\ref{propthreecases}. 

$6^0$. It is in the analog of $6^0$ that there is much more technical 
work to be done. Of the twelve cases listed in $4^0$, in three there 
is a single monomial with a negative coefficient, and this is $UV$. 
We list the coefficients of the 
monomials $UV$, $U^2$ and $V^2$ of the cases $(6,1,1)$, $(5,2,1)$ 
and $(4,2,2)$ respectively:

$$
(-10206,35721, 35721)~~~\, ,~~~\, (-20412, 61236, 35721)~~~\, ,~~~\, 
(-40824, 61236, 61236)~.$$
Everywhere in $6^0$ quadratic and biquadratic polynomials 
have negative discriminants. 
There are four cases in which exactly two monomials have negative signs,  
namely $(4,3,1)$, $(3,4,1)$, $(3,3,2)$ and $(2,4,2)$ 
in which we give only the monomials forming quadratic homogeneous 
polynomials with negative discriminants (multiplied by $1$ or $U$); we skip all other monomials (their coefficients are positive):

$$\begin{array}{l}
(-30618UV+76545U^2+35721V^2)+U(-10206UV+221130U^2+91854V^2)\\ \\
(-40824UV+81648U^2+35721V^2)+U(-81648UV+326592U^2+122472V^2)\\ \\
(-61236UV+76545U^2+61236V^2)+U(-20412UV+221130U^2+81648V^2)\\ \\ 
(-81648UV+81648U^2+61236V^2)+U(-163296UV+326592U^2+108864V^2)
\end{array}$$
In the cases $(2,5,1)$ and $(1,6,1)$ there are four and five 
negative monomials respectively. These cases 
are treated in a similar way:

$$\begin{array}{l}
(-51030UV+76545U^2+35721V^2)+U(-187110UV+391230U^2+153090V^2)\\ \\
+U^2(-245430UV+868725U^2+297270V^2)\\ \\ 
+U^3(-86670UV+1094472U^2+352350V^2)\end{array}$$
and 
$$\begin{array}{l}
(-6804UV+6804U^2+3969V^2)+U(-22680UV+28728U^2+12474V^2)\\ \\
+U^2(-29052UV+50436U^2+15849V^2)+U^4(-3252UV+24628U^2+3672V^2)\\ \\ 
+U^3(-16848UV+47088U^2+10368V^2)~.
\end{array}$$
In the case $(2,3,3)$, there are four negative monomials which we include 
in polynomials as follows:

$$\begin{array}{l}
(-91854UV+76545U^2+76545V^2)+(-59778U^2V^2+273375U^4+273375V^4)\\ \\ 
(-30618U^2V+221130U^3+221130V^3-30618UV^2)~.\end{array}$$
For the third polynomial 
in brackets its corresponding inhomogeneous polynomial 

$$-30618x^2+221130x^3+221130-30618x$$
has one negative and two complex conjugate roots. For a univariate real 
polynomial with positive leading coefficient and having only negative 
and complex conjugate roots we say that it is of {\em type P}. It is 
clear that the homogeneous polynomial corresponding to a type P univariate 
polynomial (we say that it is also of type P) is nonnegative. 

In the case $(1,5,2)$, there are seven negative monomials: 

$$\begin{array}{l}
(-102060UV+76545U^2+61236V^2)+U(-374220UV+391230U^2+136080V^2)\\ \\
+(-490860U^3V+868725U^4+10530U^2V^2+369360UV^3+79704V^4)\\ \\
+U^2(513540V^3-210600UV^2-173340U^2V+1094472U^3)\\ \\
+U^2(-215190U^2V^2+855450U^4+372915V^4)\\ \\ 
+U^4V(-64116UV+300060U^2+176760V^2)~.
\end{array}$$
The third and fourth of the polynomials in brackets are of type P 
hence nonnegative.

Finally, in the case $(1,4,3)$ we have also seven negative monomials:

$$\begin{array}{l}
(-122472UV+81648U^2+76545V^2)+(-383940U^2V^2+565056U^4+273375V^4)\\ \\ 
+(326592U^3-244944U^2V-40824UV^2+221130V^3)+\\ \\
+U(-359649U^2V^2+552096U^4+557928V^4)\\ \\
+(-75816U^3V^3+332928U^6+79065V^6)\\ \\ +U(-15066U^3V^3+126720U^6+138096V^6)~.
\end{array}$$
The third and the last two polynomials in brackets are of type P. 
\end{proof}

\section{Proof of Theorem~\protect\ref{Real}\protect\label{secprReal}}

Consider the polynomial  $P=(x+1)^{d-2}(x^2-zx+y)$, 
where the quadratic factor has no real roots,  
i.e. $\displaystyle{z^2<4y}$. Hence $y>0$ and $z>0$ 
(otherwise all coefficients of $P$ must be positive). If the polynomial $P$ defines the SP $\Sigma _{m,n,q}$, then any perturbation of $P$ with $d-2$ distinct negative roots close to $-1$ defines also the SP $\Sigma _{m,n,q}$. 
We expand $P$ in powers of $x$: 
$$P:=x^d+\Sigma^{d} _{j=1} p_jx^{d-j}~,$$ 
where $p_j = C^{j}_{d-2} - C^{j-1}_{d-2}z + C^{j-2}_{d-2}y$ with $C_{\mu}^{\nu}=0$ if $\mu <\nu$. 
The coefficients of $P$ define the SP $\Sigma _{m,n,q}$, 
so 

$$\begin{array}{l}p_j >0~~~\, {\rm for}~~~\,  j=1, \ldots, m-1~~~\, {\rm and~~for}~~~\,  j=m+n, \ldots, d~,~~~\,  
{\rm and}\\ \\ p_j <0~~~\, {\rm for}~~~\, j=m, m+1, \ldots, m+n-1~.\end{array}$$ The latter 
inequalities (combined with $z<2 \sqrt{y}$) yield:
$$ C^{j}_{d-2} + C^{j-2}_{d-2}y<C^{j-1}_{d-2}z<
2C^{j-1}_{d-2}\sqrt{y}, \quad j=m, m+1, \ldots, m+n-1~. $$ 
This means that :
\begin{equation}\label{eqA} \frac{C^{j-1}_{d-2}-\sqrt{\delta_{j-1}}}{C^{j}_{d-2}} 
< \sqrt{y} <  \frac{C^{j-1}_{d-2}+\sqrt{\delta_{j-1}}}{C^{j}_{d-2}}, 
\quad \quad \quad \end{equation}
where $\delta_{j-1} := (C^{j-1}_{d-2})^2 - C^{j-2}_{d-2}C^{j}_{d-2} > 0$.
Indeed, the polynomial $C^{j}_{d-2} 
- 2C^{j-1}_{d-2}\sqrt{y} + C^{j-2}_{d-2}y$ (quadratic in $\sqrt{y}$) has a positive 
discriminant $\delta_{j-1}$ hence its value is negative 
precisely when $\sqrt{y}$ is between its roots. Set  

\begin{equation}\label{eqC}Q^{\pm}(k):=(C^{k}_{d-2} 
\pm \sqrt{\delta_{k}} )/C^{k-1}_{d-2}~. \quad \quad \quad 
\end{equation}

\begin{lm}\label{Qeg}
One has $\displaystyle{Q^{\pm}(k)=\frac{d-k-1}{k} 
(1 \pm A(k))}$, where 

$$A(k)=\sqrt{1-\frac{k(d-k-2)}{(k+1)(d-k-1)}} 
= \sqrt{\frac{(d-1)}{(k+1)(d-k-1)}}~.$$
\end{lm}

\begin{lm}\label{Qdecr}
The quantities $Q^{\pm}(k)$ are decreasing functions 
in $k$ (for $k=1,2,\ldots,[\frac{d}{2}]$). 
\end{lm}

Lemmas~\ref{Qdecr} and \ref{Qeg} are proved after 
the proof of Theorem~\ref{Real}. 
It follows from Lemma~\ref{Qdecr} that one can find 
a value of $y$ satisfying conditions (\ref{eqA}) if 

\begin{equation}\label{eqE} 
Q^{-}(m-1) < Q^{+}(m+n-2) \quad \quad \quad 
\end{equation}
or equivalently 

\begin{equation}\label{eqaf}
 a-f<aB+fG~,
\end{equation}
where 
\begin{eqnarray*}
a= \frac{d-m}{m-1}>0\qquad \qquad \qquad \qquad\quad B\;=\; 
\sqrt{1- \frac{(m-1)(d-m-1)}{m(d-m)}} \qquad \qquad\\ \\
f=\frac{d-m-n+1}{m+n-2}>0 \qquad\qquad \qquad G\;=\;\sqrt{1- 
\frac{(m+n-2)(d-m-n)}{(m+n-1)(d-m-n+1)}}
\end{eqnarray*}
One has 
$$a-f=\frac {(n-1)(d-1)}{(m+n-2)(m-1)} >0$$ 
which permits to take squares in (\ref{eqaf}) to obtain the condition : 

\begin{equation}\label{eqF}H:=\frac{(a-f)^2-(aB)^2-(fG)^2}{2af}<GB \quad \quad \quad \end{equation}
which is equivalent to (\ref{eqE}). If $H<0$, then (\ref{eqF}) is trivially true. If $H\geq 0$, then (\ref{eqF}) is equivalent to $H^2<(GB)^2$, i.e. to (\ref{eqsuffcond}) (the latter equivalence can be proved using MAPLE). 
Theorem~\ref{Real} is proved.

\begin{proof}[Proof of Lemma \ref{Qeg}]

\begin{eqnarray*}
\delta_k&=&(C_{d-2}^k)^2-C_{d-2}^{k-1}C_{d-2}^{k+1}\\
&=& \left(\frac{(d-2)\ldots (d-k-1)}{k!}\right)^2-\frac{(d-2)
\ldots (d-k)}{(k-1)!}\cdot\frac{(d-2)\ldots (d-k-2)}{(k+1)!}\\
&=& \left ( \frac{(d-2)\ldots (d-k)(d-k-1)(k+1)}{(k+1)!} \right)^2\\ &&
\hspace{24mm}-\frac{(d-2)\ldots (d-k)k(k+1)(d-2)\ldots (d-k-2)}{((k+1)!)^2}\\
&=& \left ( \frac{(d-2)\ldots (d-k)}{(k+1)!} \right)^2
\left\{  (d-k-1)^2(k+1)^2\right. \\ &&\hspace{44mm}\left. -k(k+1) (d-k-1)(d-k-2) \right \}\\
&=& \left ( \frac{(d-2)!}{(d-k-1)!(k+1)!} \right)^2  
(k+1)(d-1)(d-k-1)~.
\end{eqnarray*} \\ \\
We substitute this expression of $\delta_k$ in (\ref{eqC}) to obtain 

\begin{eqnarray*}
Q^{\pm}(k)&=& \frac{(d-k-1)!(k-1)!}{(d-2)!}\cdot \left( 
\frac{(d-2)!}{(d-k-2)!k!} \pm \sqrt{\delta_k}\right )\\
&=& \frac{(d-k-1)}{k} \pm \frac{1}{k(k+1)}
\sqrt{(k+1)(d-1)(d-k-1)}\\
&=& \frac{(d-k-1)}{k} \pm \sqrt{\frac{(d-k-1)^2}{k^2}-
\frac{(d-k-1)(d-k-2)}{k(k+1)}}\\
&=& \frac{(d-k-1)}{k}\left( 1 \pm \sqrt{1- 
\frac{k(d-k-2)}{(k+1)(d-k-1)}} \right) ~.
\end{eqnarray*}
\end{proof}

\begin{proof}[Proof of Lemma \ref{Qdecr}]
Both factors $\displaystyle{\frac{d-k-1}{k}}$ and 
$1+A(k)$ of $Q^{+}(k)$ are decreasing in $k$ (for $k=1,2,\ldots,[\frac{d}{2}]$) hence $Q^{+}(k)$ is also decreasing. 
We represent the quantity 
$Q^{-}(k)$ in the form $$Q^{-}(k) = 
\left( \displaystyle{\frac{d-k-2}{k+1}}\right) /\left( 1+
\sqrt{\displaystyle{\frac{d-1}{(k+1)(d-k-1)}}} \right)$$
The inequality $Q^{-}(k) > Q^{-}(k+1)$ is equivalent to 
$$\begin{array}{ccc}\frac{d-k-2}{k+1} + \frac{d-k-2}{k+1} 
\sqrt{\frac{d-1}{(k+2)(d-k-2)}}&>&\frac{d-k-3}{k+2} 
+ \frac{d-k-3}{k+2} \sqrt{\frac{d-1}{(k+1)(d-k-1)}}
\end{array} $$ \\
This follows from $\displaystyle{\frac{d-k-2}{k+1} > 
\frac{d-k-3}{k+2}}$ , $\displaystyle{\frac{1}{(k+1) 
\sqrt{k+2}} > \frac{1}{(k+2) \sqrt{k+1}}}$ and \\ \\
$(d-k-2)(d-k-1) > (d-k-3)^2$. 
\end{proof}


\begin{thebibliography}{Dillo 83}

\bibitem{AlFu} A.~Albouy, Y.~Fu: Some remarks about Descartes' rule 
of signs. Elem. Math., 69 (2014), 186--194. Zbl 1342.12002, 
MR3272179







\bibitem{FoKoSh} J.~Forsg\aa rd, B.~Shapiro and V.~P.~Kostov: 
Could Ren\'e Descartes have known this?  Exp. Math. 24 (4)  
(2015), 
438-448. Zbl 1326.26027, MR3383475

\bibitem{Gr} D.~J.~Grabiner: Descartes’ Rule of Signs: 
Another Construction. Am. Math. Mon. 106 (1999), 854--856.
Zbl 0980.12001, MR1732666




\bibitem{KoCzMJ}V.~P.~Kostov, On realizability of sign patterns by real polynomials, 
Czechoslovak Math. J. 68 (143) (2018), no. 3, 853–874.

\bibitem{KoMB}V.~P.~Kostov, Polynomials, sign patterns and Descartes' rule of signs, 
Mathematica Bohemica 144 (2019), No. 1, 39-67.



\bibitem{Mitev} Mitev New inequalities between elementary 
symmetric polynomials, J. of Inequalities in Pure and Appl. 
Math. 4 (2) (2003), 11 p.

\end{thebibliography}
\end{document}